\documentclass[reqno]{amsart}
\usepackage{amsmath,amsbsy,amsfonts,enumerate}
\usepackage{color}
\date{}
\theoremstyle{plain}
\newtheorem{theorem}{Theorem}[section]

\newtheorem{lemma}[theorem]{Lemma}

\theoremstyle{definition}

\newtheorem{remark}[theorem]{Remark}
\numberwithin{equation}{section}
\numberwithin{theorem}{section}
\begin{document}

\title[Some uniqueness results]{Some uniqueness results in quasilinear subhomogeneous problems}
\vspace{1cm}
%%%%%%%%%%%%%%%%%%%%%%%%%%%%%%%%%%%%%%%%%%%%%%%%%%%%%%%%%%%%%%%%%%%%%%%%

\author{Humberto Ramos Quoirin}
\address{H. Ramos Quoirin \newline CIEM-FaMAF, Universidad Nacional de C\'{o}rdoba, (5000) C\'{o}rdoba, Argentina.}
\email{\tt humbertorq@gmail.com}


\maketitle

\begin{abstract}
We establish uniqueness results for  quasilinear elliptic problems through the criterion recently provided in \cite{DFMST}. We apply it to generalized $p$-Laplacian subhomogeneous problems that may admit multiple nontrivial nonnegative solutions. Based on a generalized hidden convexity result, we show that uniqueness holds among strongly positive solutions and nonnegative global minimizers. Problems involving nonhomogeneous operators as the so-called $(p,r)$-Laplacian are also treated.\\
\end{abstract}

\bigskip
\section{Introduction and main results}
\bigskip

A rather general criterion has been recently formulated in \cite{DFMST} to prove uniqueness results for positive critical points of a given functional. Roughly speaking, the authors take advantage of a basic convexity principle, namely, the fact that if $f:[0,1] \to \mathbb{R}$ is differentiable and strictly convex then it has at most one critical point. This principle is then applied to $f:=I \circ \gamma$, where $I$ is a given functional and $\gamma$ is a path connecting two hypotethical critical points of $I$.  This uniqueness criterion finds one of its several applications in the generalized $p$-Laplacian problem 
\[
\left\{
\begin{array}
[c]{lll}%
-\mbox{div}\left(h(|\nabla u|^p)|\nabla u|^{p-2}\nabla u\right)=g(x,u) & \mathrm{in} & \Omega,\\
u=0 & \mathrm{on} & \partial\Omega,
\end{array}
\right.  \leqno{(P)}
\]
which is the Euler-Lagrange equation associated to the functional
\begin{equation}\label{ei}
I(u):=\int_{\Omega} \left(\frac{1}{p}H(|\nabla u|^{p})-G(x,u)\right)  ,
\end{equation}
acting in the Sobolev space $W_0^{1,p}(\Omega)$. Here $p \in (1,\infty)$, $\Omega$ is a bounded and smooth domain of $\mathbb{R}^{N}$ ($N\geq1$), 
 $H(t):=\int_0^t h(s) ds$, and $G(x,t):=\int_0^t g(x,s) ds$, where\\
\begin{enumerate}
	\item[(H1)] $h:[0,\infty) \to [0,\infty)$ is continuous, bounded and nondecreasing,\\
\end{enumerate}
and $g:\Omega \times \mathbb{R} \to \mathbb{R}$ is continuous and satisfies\\
\begin{enumerate}
		\item[(A1)]  $|g(x,t)|\leq C(1+|t|^\sigma)$ for all $(x,t) \in \Omega \times \mathbb{R}$ and some $C,\sigma>0$ with $\sigma(N-p)\leq (p-1)N+p$,
	
		\item[(A2)] For every $x \in \Omega$ the map $t \mapsto \frac{g(x,t)}{t^{p-1}}$ is nonincreasing in $(0,\infty)$.\\
	
\end{enumerate}
Moreover, it is assumed that 

\begin{enumerate}
	\item[(A3)] Any nonnegative critical point of $I$ is continuous on $\overline{\Omega}$, and any two positive critical points $u,v$ of $I$ satisfy $\delta^{-1}v \leq u \leq \delta v$ in $\Omega$, for some $\delta \in (0,1)$.\\
\end{enumerate}

The next result follows from \cite[Theorem 4.1]{DFMST}:

\begin{theorem}\label{t1}
Under the above conditions on $h$ and $g$, let $A$ be the set of positive critical points of $I$. 
\begin{enumerate}
\item If $h>0$ in $(0,\infty)$ then $A \subset \{\alpha u_0: \alpha \geq 0\}$ for some $u_0 \in A$.
\item  If either $h$ is increasing or the map in $(A2)$ is decreasing, then  $A$ is at most a singleton.
\end{enumerate}
\end{theorem}

\subsection*{Notation and terminology}
Before proceeding, let us fix some terminology. We say that $f$ is a nondecreasing (respect. increasing) map on $\mathbb{R}$ if $f(t)\leq f(s)$ (respect. $f(t)< f(s)$) for any $t<s$. We say that $u \in C(\overline{\Omega})$ is nonnegative (respect. positive) if $u(x) \geq 0$ (respect. $u(x)>0$) for all $x \in \Omega$. If $u$ is a measurable function, inequalities involving $u$ are understood holding a.e., and integrals involving $u$ are considered with respect to the Lebesgue measure. A critical point is said nontrivial if it is not identically zero.
\\

Let us comment on the assumptions of Theorem \ref{t1}. First of all, (H1) and (A1)  ensure that $I$ is a $C^1$ functional whose critical points are weak solutions of $(P)$. %(and by regularity, these ones are in $C^{1,\alpha}(\Omega)$). 
Given $u,v \in W_0^{1,p}(\Omega)$ positive, set $$\gamma_p(t):=\left((1-t)u^p+tv^p)\right)^{\frac{1}{p}}, \quad t \in [0,1].$$
 Since $h$ is nondecreasing, the map $t \mapsto \int_\Omega H(|\nabla \gamma_p(t)|^p)$ is convex on $[0,1]$.\footnote{This property, known as (generalized) {\it hidden convexity}, plays an important role in several uniqueness results, and has connections with Hardy and Picone inequalities, as shown in \cite{BF}.} Moreover,   it is strictly convex if $u \not \equiv v$ and $h$ is increasing or the map in (A2) is decreasing, cf.  \cite[Lemma 4.3]{DFMST}.
 Still by (A2),  the map $t \mapsto G(x,t^{\frac{1}{p}})$ is concave on $[0,\infty)$
and consequently $I$ is strictly convex along $\gamma_p$. 
Finally, (A3) guarantees  that $\gamma_p$  is Lipschitz continuous at $t=0$ and $t=1$. It follows that if $I'(u)=I'(v)=0$ then $t \mapsto I(\gamma_p(t))$ is differentiable on $[0,1]$, with $t=0$ and $t=1$ as critical points, which is impossible for a strictly convex map.
The same contradiction arises  if $u,v$ are linearly independent and $h>0$ in $(0,\infty)$, as in this case $I$ is still strictly convex along $\gamma_p$.

\begin{remark}\label{r0}
\strut
\begin{enumerate}
	\item If $h \equiv 1$ then (A3) is a consequence of (A1) and (A2). Indeed, first note that $g(x,0) \geq 0$ in $\Omega$, since otherwise for some $x \in \Omega$ we have $\frac{g(x,t)}{t^{p-1}} \to -\infty$ as $t \to 0^+$, which is clearly impossible by (A2).
	We deduce that
	for any $t_0>0$ there exists $M>0$ such that $g(x,t)+Mt^{p-1} \geq 0$ in $\Omega \times [0,t_0]$, in which case the strong maximum principle and Hopf Lemma apply, cf. \cite{Va}.
	 In particular, the second assertion in (A3) is satisfied for any nonnegative critical points $u,v \not \equiv 0$, so that in this case Theorem \ref{t1} holds with $A$ redefined as the set of nontrivial nonnegative critical points of $I$. Note also that by (A1) any such point belongs to $C^1(\overline{\Omega})$, cf. \cite{db,L}.\\  
	
	\item When dealing with $I$ defined in $W^{1,p}(\Omega)$, Theorem \ref{t1}(2) needs the following modification:
	\begin{enumerate}
		\item If the map in $(A2)$ is decreasing then  $A$ is at most a singleton.
		\item If $h$ is increasing and $A$ contains a nonconstant element, then it is a singleton.
	\end{enumerate}
	 As a matter of fact, one may easily find some $g$ and $h$ satisfying (H1) and (A1)-(A3) with $h$ increasing and $g(x,u)\equiv 0$ for $u$ in some interval of positive constants, in which case $A$ has infinitely many elements.
	 The uniqueness result `modulo positive constants' in (b) is due to the fact that the map $t \mapsto \int_\Omega H(|\nabla \gamma_p(t)|^p)$ is identically zero and therefore not strictly convex if $u\not \equiv v$ are positive constants, see Remark \ref{r2} below.\\
\end{enumerate}
\end{remark}
Theorem \ref{t1} provides an alternative proof of some classical results for 
 \begin{equation}\label{p1}
\left\{
\begin{array}
[c]{lll}%
-\Delta_{p}u=g(x,u) & \mathrm{in} & \Omega,\\
u=0 & \mathrm{on} & \partial\Omega,\\
\end{array}
\right.  
\end{equation}
which corresponds to $(P)$ with $h \equiv 1$.
Indeed, in \cite[Examples 4.4 and 4.6]{DFMST} it is shown that
\begin{itemize}
\item $A$ is a singleton if $g(x,u)=u^{q-1}$ and $1<q<p$ (the subhomogeneous problem  \cite{DS}, or sublinear if $p=2$ \cite{BO}).
\item $A$ is  one-dimensional if $g(x,u)=\lambda_1u^{p-1}$ (the homogeneous or eigenvalue problem \cite{An}), where 
\begin{equation}\label{l1}
\lambda_1=\lambda_1(p):= \displaystyle \inf \left\{\int_\Omega |\nabla u|^p: u \in X, \int_\Omega |u|^p=1\right\}
\end{equation}
 is the first eigenvalue of the $p$-Laplacian on $X$ . In this case $A$ is the set of positive eigenfunctions associated to $\lambda_1$.\\
\end{itemize}
Theorem \ref{t1} also applies to \eqref{p1} if \begin{equation}\label{m}
g(x,u)=a(x)u^{q-1} \quad \mbox{with}\quad  1<q<p.
\end{equation}  and $a \in C(\overline{\Omega})$, $a \geq 0$, $\not \equiv 0$. On the other hand, it is clear that (A2) fails if $a$ is negative in some part of $\Omega$. In this case it is known that $(P)$ may have dead core solutions, i.e. solutions vanishing in some open subset of $\Omega$ \cite[Proposition 1.11]{D}, as well as solutions reaching some part of the boundary with null normal derivative \cite[Proposition 2.9]{KRQU16}, which shows that (A3) is not satisfied either.

Our first purpose is to deal with a class of nonlinearities including \eqref{m} with $a$  sign-changing. More precisely, instead of (A2) we shall assume that\\

\begin{enumerate}
	\item [(A2')] There exists $q \in (1,p)$  such that for a.e. $ x \in \Omega$ the map  $t \mapsto \frac{g(x,t)}{t^{q-1}}$ is nonincreasing in $(0,\infty)$.\\
\end{enumerate}

We consider the functional $I$ defined either in $W_0^{1,p}(\Omega)$ or $W^{1,p}(\Omega)$.  To overcome (A3) we restrict ourselves to strongly positive critical points, i.e. critical points lying in \[
\mathcal{P}^{\circ}:=\left\{
\begin{array}
[c]{ll}%
\left\{  u\in C_{0}^{1}(\overline{\Omega}):u>0\ \mbox{in $\Omega$},\ \partial
_{\nu}u<0\ \mbox{on $\partial \Omega$}\right\}   &
\mbox{if $X=W_0^{1,p}(\Omega)$},\medskip\\
\left\{  u\in C^{1}(\overline{\Omega}%
):u>0\ \mbox{on $\overline{\Omega}$}\right\}   &
\mbox{if $X=W^{1,p}(\Omega)$},
\end{array}
\right.
\]
 where $\nu$ is the outward unit normal to $\partial\Omega$. Such restriction is not  technical, as $I$ may have multiple nonnegative critical points if $h \equiv 1$ and $g$ is given by \eqref{m} with $a$ sign-changing, cf. \cite[Remark 1.3(3)]{KRQUpp2}. See also \cite{bandle, BPT} and \cite[Theorem 1.4(ii)]{NoDEA} and
 \cite[Proposition 5.1]{KRQU3} for $p=2$. On the other hand, $I$ has at most one positive critical point for such $h$ and $g$, cf. \cite{BPT,DS} for $p=2$ and \cite{KRQUpp} for $p>1$. Let us note that  a condition similar to (A2') with $p=2$ appears  in \cite{DS}.

Let $X=W_0^{1,p}(\Omega)$. In association with (A2'), we shall  consider the path
\begin{equation}\label{eg}
\gamma_q(t):=\left((1-t)u^q+tv^q)\right)^{\frac{1}{q}}, \quad t \in [0,1],
\end{equation}
to connect two critical points $u,v \in \mathcal{P}^{\circ}$. This path has been mostly used when $q=p$, but for $q<p$ one can find the expression $\gamma_q(1/2)=\left(\frac{u^q+v^q}{2}\right)^{\frac{1}{q}}$ in uniqueness arguments for sublinear or subhomogeneous type problems \cite{KRQUpp,KRQUpp2,KLP,N}. Note that $\gamma_q$ is  Lipschitz continuous at $t=0$ for such $u,v$ \cite[Corollary 3.3]{DFMST}.  By Lemma \ref{l1} below,  if $h>0$ in $(0,\infty)$ and $u \not \equiv v$ then $t \mapsto \int_\Omega H(|\nabla \gamma_q(t)|^p)$ is stricly convex on $[0,1]$. (A2') provides then the strict convexity of $t \mapsto I(\gamma_q(t))$. 
Furthermore, when dealing with global minimizers of $I$, uniqueness holds not only within $\mathcal{P}^{\circ}$, but more generally among nonnegative functions. For such $u,v$ the path
$\gamma_q$ is not necessarily Lipschitz continuous at $t=0$, yet $I$ is stricly convex along $\gamma_q$ for $u \not \equiv v$. In particular, the inequality \begin{equation}\label{ii}
I(\gamma_q(t))<(1-t)I(u)+tI(v)
\end{equation}
holds for any $t \in (0,1)$, and
the global minimality of $I(u)=I(v)$ eventually yields a contradiction. 
Lastly, let us mention that if $X=W^{1,p}(\Omega)$ then the above arguments hold assuming moreover that either $u$ or $v$ is nonconstant.

Our second purpose is to deal with problems involving nonhomogeneous operators, as $-\Delta_p-\Delta_r$ with $1<p<r$, in which case the associated functional is
$$I(u)=\int_\Omega \left( \frac{1}{p}|\nabla u|^p + \frac{1}{r}|\nabla u|^r -G(x,u)\right),$$
defined on $W^{1,r}(\Omega)$ or $W_0^{1,r}(\Omega)$. Some uniqueness results have been proved in \cite{BT,KST,MP,T} for this class of problems under Dirichlet boundary conditions. Note that this functional corresponds to \eqref{ei} with $h(t)=1+t^{\frac{r}{p}-1}$. 

More generally, we shall assume that\\
\begin{enumerate}
	\item[(H1')] $h:[0,\infty) \to [0,\infty)$ is continuous, nondecreasing, and $h(t)\leq C( 1+t^{\frac{r}{p}-1})$ for some $C>0$, $r>p$, and any $t \geq 0$,\\
\end{enumerate}
and, instead of $(A1)$, \\
\begin{enumerate}
	\item[(A1')] $|g(x,t)|\leq C(1+|t|^\sigma)$ for all $(x,t) \in \Omega \times \mathbb{R}$ and some $C,\sigma>0$ with  $\sigma(N-r)\leq (r-1)N+r$.\\
\end{enumerate}

It is clear that $I$ given by \eqref{ei} is a $C^1$ functional on $W^{1,r}(\Omega)$ under (H1') and (A1'), and these conditions are weaker than (H1) and (A1), respectively.
Finally, in view of Remark \ref{r0}(1) we strengthen $(A3)$ as follows:\\

\begin{enumerate}
	\item[(A3')] Any nonnegative critical point of $I$ is continuous on $\overline{\Omega}$, and any two such points $u,v \not \equiv 0$  satisfy $0<\delta^{-1}v \leq u \leq \delta v$ in $\Omega$, for some $\delta \in (0,1)$.\\
\end{enumerate}

%Note that if $(A3)$ holds then $A \setminus \{0\}=A \cap \mathcal{P}^{\circ}$.

The next subsection contains the statement of our results, which are proved in section 2. We apply our results to a few particular problems and recall some known uniqueness results  in section 3.\\

\subsection{Statement of our results}
From now on $I$ is the functional given by \eqref{ei}, and $A$ and $B$ are the sets of nontrivial nonnegative critical points and nonnegative global minimizers of $I$, respectively.

\begin{theorem}\label{tp}
	Let $h$ satisfy $(H1)$ and $h>0$ in $(0,\infty)$, and $g:\Omega \times \mathbb{R} \to \mathbb{R}$ be a Carath\'eodory function satisfying $(A1)$ and $(A2')$. 
\begin{enumerate}
	\item If $X=W_0^{1,p}(\Omega)$ then $A \cap \mathcal{P}^{\circ}$ and $B$ are at most a singleton.
	\item If $X=W^{1,p}(\Omega)$ and $A \cap \mathcal{P}^{\circ} $ contains a nonconstant element then it is a singleton. The same conclusion holds for $B$.
\end{enumerate}	
\end{theorem}	

\begin{remark}
	\strut
	\begin{enumerate}
		
		\item As already observed, the assertions on $B$ rely essentially on \eqref{ii}, rather than the strict convexity of $I$ along $\gamma$. Even more, \eqref{ii} satisfied at some $t \in (0,1)$ suffices to reach a contradiction. This inequality with $t=1/2$ has been used in \cite{KRQUpp,KRQUpp2} to show that $B$ is a singleton if $h \equiv 1$ and $g(x,u)=\lambda u^{p-1}+a(x)u^{q-1}$, with $q\in (1,p)$, $\lambda \leq 0$ and $a$ sign-changing. \\
		
		\item One can extend the results on $B$ to deal with local minimizers as follows: assume that $C:=\{w \in X: I'(w)=0, I(w)=c\}$ contains a local minimizer $u$ (nonconstant if $X=W^{1,p}(\Omega$)) for some $c \in \mathbb{R}$. If for any $v \in C$ with $u \not \equiv v$ the inequaity \eqref{ii} holds for every $t \in (0,1)$, then $C=\{u\}$. This can be seen as a particular case of  \cite[Theorem 2.2]{DFMST}, without requiring $\gamma$ to be Lipschitz continuous at $t=0$.\\

		\item As discussed in Remark \ref{r0}(2), the assumption that $A \cap \mathcal{P}^{\circ} $ contains a nonconstant element is natural when $X=W^{1,p}(\Omega)$. One can remove this condition assuming in addition that the map in $(A2')$ is decreasing, or that $g(x,c) \not \equiv 0$ for any $c>0$. Note also that if $X=W^{1,p}(\Omega)$ and $g$ satisfies $(A1)$, $(A2')$, and $g(x,c) \equiv 0$ for exactly one constant $c>0$, then $A \cap \mathcal{P}^{\circ} =\{c\}$. Indeed, in such case there is no other positive constant solving $(P)$, and by Theorem \ref{tp} we deduce that $I$ has no further critical point in $A \cap \mathcal{P}^{\circ}$. This situation occurs for instance if $g(x,u)=u^{q-1}-u^{r-1}$ with $q \in (1,p)$ and $r>q$, in which case $A \cap \mathcal{P}^{\circ} =\{1\}$.\\
		
	\item Under the assumptions of Theorem \ref{tp} the sets $A \cap \mathcal{P}^{\circ}$ and $B$ may be empty, as we shall see in section 3. However, when $X=W_0^{1,p}(\Omega)$ the set $B$ is nonempty if we assume moreover that $h$ is bounded away from zero and $g(x,t)$ is odd with respect to $t$ or $g(x,t)=0$ for $t<0$. Indeed, from $(A2')$ we deduce that $G(x,t) \leq C_2(1+|t|^q)$ for some $C_2>0$ and every $t \in \mathbb{R}$, so that
	$I$ is coercive, weakly lower-semicontinuous and therefore has a global minimum on $W_0^{1,p}(\Omega)$. Finding sufficient conditions to have $A \cap \mathcal{P}^{\circ}$ nonempty is a more delicate issue.  Such conditions have been provided for $g(x,u)=\lambda u^{p-1}+a(x)u^{q-1}$, with $1<q<p$, $\lambda \leq 0$ and $a$ sign-changing in \cite{KRQUpp,KRQUpp2}, see section 3.
	\end{enumerate}		
\end{remark}

Next we consider $h$ and $g$ under (H1') and (A1'). First we observe that \cite[Theorem 4.1]{DFMST} can be easily extended to this situation. Let us recall that $I$ given by \eqref{ei} acts now on $W^{1,r}(\Omega)$ or $W_0^{1,r}(\Omega)$.

\begin{theorem}\label{tp1}
	Let $h$ satisfy {\rm (H1')}, and
	$g:\Omega \times \mathbb{R} \to \mathbb{R}$ be a Carath\'eodory function satisfying {\rm (A1'), (A2)}, and {\rm (A3')}.
	\begin{enumerate}
		\item If $h>0$ in $(0,\infty)$ then $A \subset \{\alpha u_0: \alpha \geq 0\}$ for some $u_0 \in A$.
		\item  If the map in $(A2)$ is decreasing then  $A$ is at most a singleton.
		\item If $h$ is increasing and $A$ contains a nonconstant element, then it is a singleton.
	\end{enumerate}
\end{theorem}

The extension of Theorem \ref{tp} reads as follows:

\begin{theorem}\label{tp2}
	Let $h$ satisfy {\rm (H1')} and $h>0$ in $(0,\infty)$, and $g:\Omega \times \mathbb{R} \to \mathbb{R}$ be a Carath\'eodory function satisfying {\rm (A1')} and {\rm (A2')}.
	\begin{enumerate}
		\item If $X=W_0^{1,r}(\Omega)$ then $A \cap \mathcal{P}^{\circ}$ and $B$ are at most a singleton.
		\item If $X=W^{1,r}(\Omega)$ and $A \cap \mathcal{P}^{\circ} $ contains a nonconstant element then it is a singleton. The same conclusion holds for $B$.
	\end{enumerate}	
\end{theorem}	

\begin{remark}
In the same way as in Theorem \ref{tp},  the condition that $A$ contains  a nonconstant element can be removed from Theorem \ref{tp1}(3) if the map in $(A2)$ is decreasing or  $g(x,c) \not \equiv 0$ for any $c>0$. A similar remark applies to $A \cap \mathcal{P}^{\circ}$ and $B$ in Theorem \ref{tp2}.
\end{remark}

Finally, let us mention that in view of their variational nature, Theorems \ref{tp}, \ref{tp1}, and \ref{tp2} can be adapted to obtain uniqueness results for
the equation 
$$-\mbox{div}\left(h(|\nabla u|^p)|\nabla u|^{p-2}\nabla u\right)=g(x,u) \quad \mbox{in} \quad \Omega$$
under nonlinear or mixed boundary conditions (as shown in \cite[Examples 4.7 and 4.8]{BF}).
We refer to \cite{MS} for similar results in the semilinear case $h \equiv 1$ and $p=2$.\\

\section{Proofs}
\medskip

The proofs are based on \cite[Theorem 1.1]{DFMST} and the following lemma, which sharpens the {\it general hidden convexity} established in \cite[Proposition 2.6]{BF} by analyzing the strict convexity of the map $t \mapsto \int_\Omega H(|\nabla \gamma_q(t)|^p)$. Moreover $H$ is not assumed to be homogeneous:
\begin{lemma}\label{l1}
	Let $1<q<p$, and $u,v \in W^{1,p}(\Omega)$ with  $u,v \geq 0$ in $\Omega$. \\
	\begin{enumerate}
		\item Let $\gamma_q$ be given by \eqref{eg}. Then \begin{equation}\label{gq}
		|\nabla \gamma_q(t)|^p \leq (1-t)|\nabla u|^p + t |\nabla v|^p \quad \forall t \in [0,1].
		\end{equation}
		Moreover, for $t \in (0,1)$ the strict inequality holds in the set\\
		\[
		\tilde{Z}(u,v):=\{x\in\Omega: u(x)\neq v(x),|\nabla u(x)|+|\nabla v(x)|>0\}.
		\]\\
		
		\item Let $h$ satisfy  $h>0$ in $(0,\infty)$ and (H1') . If $u,v \in X=W_0^{1,p}(\Omega)$ and $u \not \equiv v$ then the map $t \mapsto \int_\Omega H(|\nabla \gamma_q(t)|^p)$ is stricly convex on $[0,1]$. In particular,\\
		$$\int_\Omega H(|\nabla \gamma_q(t)|^p) \leq (1-t)\int_\Omega H(|\nabla u|^p)+t\int_\Omega H(|\nabla v|^p) \quad \forall t \in [0,1],$$
		with strict inequality for $t \in (0,1)$.
		The same assertions hold for $X= W^{1,p}(\Omega)$ assuming in addition that either  $u$ or $v$ is nonconstant.
	\end{enumerate}
\end{lemma}

\begin{proof}\strut
	\begin{enumerate}
		\item  It is enough to prove \eqref{gq} in the  set
		\[
		Z=Z(u,v):=\{x\in\Omega: |\nabla u(x)|+|\nabla v(x)|>0\},
		\]
		since outside $Z$ one has
		$\nabla \gamma_q(t)=0$. 
		We claim that \cite[Lemma 3.5]{DFMST} applies with $Q(t)=t^q$ and $M(t)=t^p$. Indeed, in this case $F_1(t)=qt^{1-\frac{1}{q}}$ and $F_2(t)=t^{\frac{1}{p}}$. One may easily check that $F(z_1,z_2):=F_1(z_1)F_2(z_2)$ is concave in $[0,\infty) \times [0,\infty)$ and strictly concave in $(0,\infty) \times (0,\infty)$. Arguing as in the proof of (3.9) in \cite[Lemma 3.5]{DFMST} we find that
		$$F((1-t)z+t\bar{z})>(1-t)F(z)+tF(\bar{z})$$
		if $t \in (0,1)$, $z \in [0,\infty) \times [0,\infty)$, and $\bar{z} \in (0,\infty) \times (0,\infty)$ with $z \neq \bar{z}$.
		Applying this inequality to $z=(Q(u),M(|\nabla u|))$ and $\bar{z}=(Q(v),M(|\nabla v|))$ 	we infer that \eqref{gq} holds with strict inequality in $\tilde{Z}(u,v)$.\\
		\item  Let $t,s, \alpha \in [0,1]$. Note that
		$$\gamma_q((1-\alpha)t+\alpha s)=\left((1-\alpha) \gamma_q(t)^q+\alpha  \gamma_q(s)^q\right)^{\frac{1}{q}},$$
		so applying \eqref{gq} with $u,v$ replaced by $\gamma_q(t), \gamma_q(s)$ respectively, we find that\\
		$$ |\nabla \gamma_q((1-\alpha)t+\alpha s)|^p \leq (1-\alpha)  |\nabla \gamma_q(t)|^p+\alpha |\nabla \gamma_q(s)|^p,$$
		with strict inequality in $\tilde{Z}(\gamma_q(t),\gamma_q(s))$.
		Since $H$ is nondecreasing and convex, we have\\
		$$\int_\Omega H(|\nabla \gamma_q((1-\alpha)t+\alpha s)|^p) \leq (1-\alpha) \int_\Omega H(|\nabla \gamma_q(t)|^p)+\alpha \int_\Omega H(|\nabla \gamma_q(s)|^p).$$
		\\
		\noindent Assume by contradiction that equality holds for some $t\neq s$ and $\alpha \in (0,1)$. Note that $h>0$ yields that $H$ is increasing in $(0,\infty)$, so that such equality holds only if $\tilde{Z}(\gamma_q(t),\gamma_q(s))$ is null. It follows that $\nabla \gamma_q(t)\equiv \nabla  \gamma_q(s)$, so 
		$\gamma_q(t)\equiv \gamma_q(s) +c$ for some $c \in \mathbb{R}$. If $X=W_0^{1,p}(\Omega)$ then $c=0$ and $u\equiv v$. If $X=W^{1,p}(\Omega)$ then we can assume that $c>0$, i.e. $\gamma_q(t)> \gamma_q(s)$. %Hence $(s-t)(u^q-v^q)>0$, ie. $u^q-v^q$ has a sign. 
		It follows that $|\nabla \gamma_q(t)|+|\nabla  \gamma_q(s)|=0$ a.e., so that $\gamma_q(t), \gamma_q(s)$ are constant functions. Consequently $u,v$ are constant, and we reach a contradiction.
	\end{enumerate}	
\end{proof}

\begin{remark}\label{r2}
\strut
\begin{enumerate}
\item Lemma \ref{l1}(1) has been proved for $t=1/2$ in  \cite{KLP}. Using this inequality the authors show that the set of nonnegative minimizers for the Rayleigh quotient
$\frac{\int_\Omega |\nabla u|^p}{\left(\int_\Omega a(x)|u|^q\right)^{\frac{p}{q}}}$ in $W_0^{1,p}(\Omega)$ is one-dimensional if $1<q\leq p$, with $a$ allowed to change sign.\\

\item As shown in \cite[Lemmas 3.9 and 4.3]{DFMST}, the inequality \eqref{gq} and the convexity of $t \mapsto \int_\Omega H(|\nabla \gamma_q(t)|^p)$ still hold for $q=p$. However, this map is strictly convex if either $h>0$ and $u,v$ are linearly independent or $h$ is increasing and $u \not \equiv v$, with one of them nonconstant when $X=W^{1,p}(\Omega)$. Indeed, as shown in the proof of \cite[Lemma 4.3]{DFMST}(2), if $t \mapsto \int_\Omega H(|\nabla \gamma_p(t)|^p)$ is not strictly convex then $u,v$ are linearly dependent, and for $h$ increasing we deduce that $|\nabla \gamma_p(t_1)|=|\nabla \gamma_p(t_2)|$ for some $t_1\neq t_2$. It follows that $u \equiv v$ or $\nabla u\equiv \nabla v \equiv 0$, i.e. $u \equiv v$ or $u,v$ are positive constants.\\
\end{enumerate}	
\end{remark}

\begin{proof}[{\bf Proof of Theorem \ref{tp}}]
	Let us first prove the assertions on $A \cap \mathcal{P}^{\circ}$. We adapt the proof of \cite[Theorem 4.1]{DFMST}, which is based on \cite[Theorem 1.1]{DFMST}. Assume by contradiction that $u,v \in \mathcal{P}^{\circ}$ are critical points of $I$ with $u \not \equiv v$. First note that $\gamma_q$ is is locally Lipschitz at $t=0$ \cite[Corollary 3.3]{DFMST} since $u,v$ are comparable in the sense of (3.2), as shown in \cite[Lemma 3.4]{DFMST}.
	Assuming that $u$ is nonconstant if $X=W^{1,p}(\Omega)$, Lemma \ref{l1} yields that $t \mapsto \int_\Omega H(|\nabla \gamma_q(t)|^p)$ is stricly convex on $[0,1]$.  (A2') yields that $t \mapsto G(x,t^{\frac{1}{q}})$  is concave  on  $[0,\infty)$, so $t \mapsto I(\gamma_q(t))$ is strictly convex, and its derivative vanishes at $t=0$ and $t=1$, which is impossible.
	
	Let us now deal with $B$. Assume that $u,v \in B$ and $u \not \equiv v$ (with at least one of them  nonconstant when $X=W^{1,p}(\Omega)$). By Lemma \ref{l1} the map $t \mapsto I(\gamma_q(t))$ is stricly convex on $[0,1]$, so
	\eqref{ii} holds for any $t \in (0,1)$,
	which is clearly impossible since $u,v$ minimize $I$ globally. 
\end{proof}

\begin{proof}[{\bf Proof of Theorem \ref{tp1}}]
One can repeat the proof of \cite[Theorem 4.1]{DFMST} with the same path $\gamma_p$, and replacing $W^{1,p}(\Omega)$ by $W^{1,r}(\Omega)$. By (H1') and $(A1')$ the functional $I$ is $C^1$ on $W^{1,r}(\Omega)$. Note also that \cite[Lemma 4.3(ii)]{DFMST} holds in $W^{1,r}(\Omega)$ (or $W^{1,p}(\Omega)$) if we assume in addition that either $u$ or $v$ is nonconstant, cf. Remark \ref{r2}.
\end{proof}	

\begin{proof}[{\bf Proof of Theorem \ref{tp2}}]
	By (H1') we have $0\leq H(t) \leq C(1+t^{\frac{r}{p}})$ for some $C>0$, and any $t \geq 0$. Note  that \eqref{gq} holds also with $p$ replaced by $r$, so that Lemma \ref{l1}(2) holds in particular for $u,v \in X$, with $X=W^{1,r}(\Omega)$ or $X=W_0^{1,r}(\Omega)$. We can then repeat the proof of Theorem \ref{tp} to get the desired conclusions.
\end{proof}

\medskip
\section{Examples}
\medskip

Next we apply our results to some specific problems. Below we assume that $a,b \in L^{\infty}(\Omega)$, and $p^*$ is the critical Sobolev exponent, i.e. $p^*=Np/(N-p)$ if $p<N$ and $p^*=\infty$ if $p \geq N$.
\begin{enumerate}
\item First we apply Theorem \ref{tp} to $h \equiv 1$, i.e. to the problem \eqref{p1}. Let us observe that in the following cases the strong maximum principle and the Hopf lemma do not apply.\\
\begin{enumerate}
\item $g(x,u)=a(x)u^{q-1}$ with $q\in (1,p)$ and $a $ sign-changing.\\

 Note that $g(x,c)\not \equiv 0$ for any constant $c>0$, so we deduce that $A \cap \mathcal{P}^{\circ}$ and $B$ are at most a singleton if $X=W^{1,p}(\Omega)$ or $X=W_0^{1,p}(\Omega)$.

 When $p=2$ and $a$ is smooth, the fact that $A \cap \mathcal{P}^{\circ}$ is at most a singleton has been proved  in \cite{BPT} for $X=H^1(\Omega)$, and in \cite{DS} for $X=H_0^1(\Omega)$. These results hold for classical solutions, and are proved via a change of variables and the maximum principle.
 In  \cite[Theorem 1.1]{KRQUpp} it is shown, for $p>1$, that $B$ is a singleton (assuming that $\int_\Omega a<0$ if $X=W^{1,p}(\Omega)$) and $A \cap \mathcal{P}^{\circ} \subset B$. It is also shown in \cite[Theorem 1.2 and Remark 2.10]{KRQUpp} that $A \cap \mathcal{P}^{\circ}$ is a singleton if either $q$ is close enough to $p$ or $X=W_0^{1,p}(\Omega)$ and $a^-$ is sufficiently small. When $a$ is too negative in some part of $\Omega$ every nonnegative critical point has a dead core \cite[Proposition 1.11]{D}, i.e. a region where it vanishes. In this situation no positive solution exists, so that $A \cap \mathcal{P}^{\circ}$ is empty, but $B$ is a singleton. Finally, both $A \cap \mathcal{P}^{\circ}$ and $B$ are empty when $\int_\Omega a>0$ and $X=W^{1,p}(\Omega)$. Indeed, in this case it is clear that the functional is unbounded from below (over the set of constant functions), and taking $u^{1-q}$ as test function we see that $\int_\Omega a<0$ if $u \in A \cap \mathcal{P}^{\circ}$.
 \\   %We refer to \cite[Theorem 2.1 and Proposition 2.2]{DS},  \cite[Lemma 3.1]{BPT}, and \cite[Theorem 1.1]{KRQUpp} for the precise statements of these results.
\item $g(x,u)=a(x)u^{q-1}+b(x) u^{r-1} $ with $q\in (1,p)$, $r\in (q,p^*)$,  $a$ sign-changing and $b \leq 0$, $\not \equiv 0$.\\

For such $g$ one may easily show that $I$ is coercive.  Moreover $g(x,c) \not \equiv 0$ for any constant $c>0$, so that $B$ is a singleton and $A \cap \mathcal{P}^{\circ}$ is at most a singleton.
 
Following the approach of \cite{BPT}, the fact that $A \cap \mathcal{P}^{\circ}$ is at most a singleton has been proved 
in \cite{alama} for $X=H^1(\Omega)$, $r=p=2$, and $b \equiv \lambda<0$. In \cite{KRQUpp2} the results of \cite{KRQUpp} were extended to any $r=p$ and $b \equiv \lambda<0$. In this case $B$ is a singleton \cite[Theorem 1.1]{KRQUpp}, whereas $A \cap \mathcal{P}^{\circ}$ is a singleton if either $q$ is close enough to $p$ or $a^-$ is sufficiently small  \cite[Theorem 1.6 and Proposition 4.6]{KRQUpp}, and it is empty if $a$ is negative enough in some part of $\Omega$  \cite[Theorem 1.8]{KRQUpp}.  When $X=W^{1,p}(\Omega)$, $\int_\Omega a<0$, and $q$ is close to $p$, the condition $\lambda \leq 0$ is optimal for the uniqueness in $A \cap \mathcal{P}^{\circ}$, as  $I$ has at least two critical points  in $\mathcal{P}^{\circ}$ for $\lambda>0$ small  \cite[Theorem 1.4]{KRQUpp}. Finally, it is not difficult to see that $A \cap \mathcal{P}^{\circ}$ and $B$ are empty if $\lambda>0$ is large enough.

Let us note that in \cite{Dpp} it is claimed that $B$ is a singleton for $r=2q$ and $b \equiv -1$, but we have not found a proof of this uniqueness result.\\
 
\item $g(x,u)=a(x)u^{q-1}+b(x) u^{r-1}$ with $1\leq r<q<p$,  $a^- \not \equiv 0$, and $b \geq 0$, $\not \equiv 0$.\\

In this case $B$ is a singleton if we assume that $\int_\Omega a<0$ when $X=W^{1,p}(\Omega)$. If $a$ is sign-changing then $A \cap \mathcal{P}^{\circ}$ is a singleton when $q$ is close enough to $p$ or  $X=W_0^{1,p}(\Omega)$ and $a^-$ is sufficiently small. Indeed, as mentioned in (1), in this case the equation $-\Delta_p u =a(x)u^{q-1}$ has exactly one solution in $\mathcal{P}^{\circ}$, and since $b\geq 0$ this solution is a subsolution for $-\Delta_p u=g(x,u)$. Since this problem has arbitrarily large supersolutions, we obtain a solution in $\mathcal{P}^{\circ}$ by the sub-supersolutions method, so that  $A \cap \mathcal{P}^{\circ}$ is a singleton.\\
\end{enumerate}
\item We apply now Theorems \ref{tp1} and \ref{tp2} to $h(t)=1+t^{\frac{r}{p}-1}$, with $1<p<r$. Note that for such $h$ and $g$ satisfying $(A1')$, critical points of $I$, defined on $W^{1,r}(\Omega)$ or $W_0^{1,r}(\Omega)$, belong to $C^1(\overline{\Omega})$, cf. \cite{L2}.\\

\begin{enumerate}
\item $g(x,u)=a(x)u^{p-1}-b(x)u^{q-1}$, with $q>p$ and $b \geq 0$. \\

Let us first assume $q<r^*$. The strong maximum principle and the Hopf lemma  \cite[Theorems 5.3.1 and 5.5.1]{PS} apply to $(P)$, so (A3') is satisfied. Note also that $g(x,c) \equiv 0$ for at most one constant $c>0$ (and in this case $a,b$ are linearly dependent).
By Theorem \ref{tp1} we deduce that $I$ has at most one nontrivial nonnegative critical point. Since $p<r$ and $b \geq 0$, it follows that $I$ is coercive, assuming if $X=W^{1,r}(\Omega)$ that $b \not \equiv 0$ or $\int_\Omega a<0$. It has a nontrivial global minimizer if $a(x) \geq \lambda_1(p)$, where $\lambda_1(p)$ is given by \eqref{l1}.
More generally, this result holds if  $\displaystyle \inf_{u \in X} \int_\Omega \left(|\nabla u|^p -a(x)|u|^p\right)<0$,  which is also a necessary condition to have $A \neq \emptyset$.
Now, for $q \geq r^*$ one may argue as in \cite[Example 4.9]{DFMST} to show that $A$ is a singleton if $a \equiv a_0 >\lambda_1(p)$ and $b \equiv b_0 > 0$. Note that if $X=W^{1,r}(\Omega)$ then $A=\{(a_0b_0^{-1})^{\frac{1}{q-p}}\ \}$.

For $a>\lambda_1(p)$, $X=W_0^{1,r}(\Omega)$, and $q=r$, the existence and uniqueness of a positive critical point has been observed  in \cite{BT,MP,T}.\\

\item $g(x,u)=a(x)u^{q-1}$ with $1<q<p$ and $a^+ \not \equiv 0$.\\

If $a \geq 0$ then the strong maximum principle and Hopf Lemma apply, so (A3') holds and we can apply Theorem \ref{tp1} to deduce that $A$ is at most a singleton. If $X=W_0^{1,r}(\Omega)$ then $I$ is clearly coercive and has a nontrivial global minimizer, so $A$ is a singleton. Now, if $X=W^{1,r}(\Omega)$ then integrating the equation we see that $A$ is empty.

The existence and uniqueness of a positive critical point has been proved in \cite{KST} for the Dirichlet problem with $N=1$ and $a$ a positive constant.

Assume now that $a$ changes sign. %in which case the strong maximum principle does not apply, so $B \not \subset A \cap \mathcal{P}^{\circ}$. 
By Theorem \ref{tp2} the sets  $A \cap \mathcal{P}^{\circ}$ and $B$ are at most a singleton (note that no positive constant solves $(P)$). Since $q<p<r$, it is straightforward that $I$ is coercive and has a nontrivial global minimizer (so $B$ is a singleton), assuming in addition that $\int_\Omega a<0$ if $X=W^{1,r}(\Omega)$. \\
\end{enumerate}
\end{enumerate}

%\section{Nonlinear boundary conditions}


\begin{thebibliography}{99}                                                                                               %


%\bibitem {AP}B. Abdellaoui, I. Peral, \textit{Existence and nonexistence
%results for quasilinear elliptic equations involving the p-Laplacian}, Ann.
%Mat. Pura. Appl. \textbf{182} (2003) 247--270.


\bibitem {alama}S. Alama, \textit{Semilinear elliptic equations with sublinear
indefinite nonlinearities}, Adv. Differential Equations \textbf{4} (1999), 813--842.

%\bibitem {AT}S. Alama, G. Tarantello, \textit{On semilinear elliptic equations
%with indefinite nonlinearities}, Calc. Var. Partial Differential Equations
%\textbf{1} (1993), 439--475.

%\bibitem {AH}W. Allegretto, Y.X. Huang, \textit{A Picone's identity for the
%p-Laplacian and applications.} Nonlinear Anal. \textbf{32} (1998), 819--830.


%\bibitem {ALG}H. Amann, J. L\'{o}pez-G\'{o}mez, \textit{A priori bounds and
%multiple solutions for superlinear indefinite elliptic problems},
%J.\ Differential Equations \textbf{146} (1998), 336--374.

%\bibitem {ABC}A. Ambrosetti, H. Brezis, G. Cerami, \textit{Combined effects of
%concave and convex nonlinearities in some elliptic problems}, J. Funct. Anal.
%\textbf{122} (1994), 519-543.

\bibitem{An} A. Anane, {\it Simplicit\'e et isolation de la premi\`ere valeur propre du p-laplacien avec poids (French, with English summary)}, C. R. Acad. Sci. Paris Ser. I Math. 305 (1987), no. 16, 725--728.
%\bibitem {AR}D. Arcoya, D. Ruiz, \textit{The Ambrosetti-Prodi problem for the
%p-Laplacian operator}, Comm. Partial Differential Equations \textbf{31}
%(2006), 849--865.

\bibitem {bandle}C. Bandle, M. Pozio, A. Tesei, \textit{The asymptotic
behavior of the solutions of degenerate parabolic equations,} Trans. Amer.
Math. Soc. \textbf{303} (1987), 487--501.\textit{\ }

\bibitem {BPT}C.\ Bandle, M.\ Pozio, A.\ Tesei, \textit{Existence and
uniqueness of solutions of nonlinear Neumann problems},
Math.\ Z.\ \textbf{199} (1988), 257--278.

%\bibitem {BK}M. Belloni, {B.\ Kawohl, }\textit{A direct uniqueness proof for
%equations involving the }$p$\textit{-Laplace operator}, Manuscripta Math.
%\textbf{109} (2002), 229--231.

%\bibitem {BCN}H. Berestycki, I. Capuzzo-Dolcetta, L. Nirenberg,
%\textit{Superlinear indefinite elliptic problems and nonlinear Liouville
%theorems.} Topol. Methods Nonlinear Anal. 4 (1994), no. 1, 59\^{a} ``78.

%\bibitem {BCN2}H. Berestycki, I. Capuzzo-Dolcetta, L. Nirenberg,
%\textit{Variational methods for indefinite superlinear homogeneous elliptic
%problems}, NoDEA Nonlinear Differ. Equ. Appl. \textbf{2} (1995), 553--572.

%\bibitem {BD}I. Birindelli, F. Demengel, \textit{Existence of solutions for
%semi-linear equations involving the p-Laplacian: the non coercive case.} Calc.
%Var. Partial Differential Equations \textbf{20} (2004), 343--366.

\bibitem {BT}V. Bobkov, M. Tanaka, {\it Remarks on minimizers for (p,q)-Laplace equations with two parameters.} Communications on Pure and Applied Analysis,  17(3), 1219--1253.

\bibitem{DFMST} D. Bonheure, J. Foldes, E. Moreira dos Santos, A. Salda\~na, H. Tavares, {\it Paths to uniqueness of critical points and applications to partial differential equations.} Trans. Amer. Math. Soc. 370 (2018), no. 10, 7081--7127.

%\bibitem {BGH}D. Bonheure, J. M. Gomes, P. Habets, \textit{ Multiple positive
%solutions of superlinear elliptic problems with sign-changing weight}, J.
%Differential Equations \textbf{214} (2005), 36--64.


\bibitem {BF}L. Brasco, G. Franzina, \textit{ Convexity properties of
Dirichlet integrals and Picone-type inequalities}, Kodai Math. J. \textbf{37}
(2014), 769--799.

\bibitem {BO}H. Brezis, L. Oswald, \textit{ Remarks on sublinear elliptic
equations,} Nonlinear Anal. \textbf{10} (1986), 55--64.

%\bibitem {B}K.J. Brown, \textit{The Nehari manifold for a semilinear elliptic
%equation involving a sublinear term.} Calc. Var. Partial Differential
%Equations \textbf{22} (2005), 483--494.
%\bibitem {d}D. G. de Figueiredo, Positive solutions of semilinear elliptic
%equations, Lect. Notes Math. Springer \textbf{957}, 34--87 (1982).


%\bibitem {Dam}L. Damascelli, \textit{Comparison theorems for some quasilinear
%degenerate elliptic operators and applications to symmetry and monotonicity
%results}, Ann. Inst. H. Poincar\'{e} Anal. Non Lin\'{e}aire \textbf{15}
%(1998), 493--516.

\bibitem {DSu}M. Delgado, A. Su\'{a}rez, \textit{On the uniqueness of positive
solution of an elliptic equation}, Appl. Math. Lett. \textbf{18} (2005), 1089--1093.

\bibitem{Dpp}J. I. Díaz, {\it New applications of monotonicity methods to a class of non-monotone parabolic quasilinear sub-homogeneous problems,} to appear in Journal Pure and Applied Functional Analysis. Special Issue dedicated to Haïm Brezis.

\bibitem {D}J. I. D\'{\i}az, \textit{Nonlinear Partial Differential Equations
and Free Boundaries. Vol.I. Elliptic equations.} Research Notes in Mathematics
\textbf{106}, Pitman, London, 1985, 323 pp.

\bibitem {DS}J. I. D\'{\i}az, J.E. Saa, \textit{Existence et unicit\'{e}\ de
solutions positives pour certaines \'{e}quations elliptiques
quasilin\'{e}aires. (French) [Existence and uniqueness of positive solutions
of some quasilinear elliptic equations] }C. R. Acad. Sci. Paris S\'{e}r. I
Math. \textbf{305} (1987), 521--524.

\bibitem {db}E. DiBenedetto, \textit{$C^{1+\alpha}$ local regularity of weak
	solutions of degenerate elliptic equations}, Nonlinear Anal. \textbf{7}
(1983), 827--850.

%\bibitem {cuesta}M. Cuesta, P. Tak\'{a}\v{c}, \textit{A strong comparison
%principle for positive solutions of degenerate elliptic equations},
%Differential Integral Equations \textbf{13} (2000), 721--746.


%\bibitem {DFG}D. G. de Figueiredo, J.-P. Gossez, \textit{On the first curve of
%the Fucik spectrum of an elliptic operator}, Differential Integral Equations
%\textbf{7} (1994), 1285-1302.


%\bibitem {DGU1}D. G. de Figueiredo, J-P. Gossez, P. Ubilla, \textit{Local
%superlinearity and sublinearity for indefinite semilinear elliptic problems},
%J. Funct. Anal. \textbf{199} (2003), 452--467.


%\bibitem {DGU2}D. G. De Figueiredo, J-P. Gossez, P. Ubilla,
%\textit{Multiplicity results for a family of semilinear elliptic problems
%under local superlinearity and sublinearity}, J. Eur. Math. Soc. \textbf{8}
%(2006), 269--286.


%\bibitem {db}E. DiBenedetto, \textit{$C^{1+\alpha}$ local regularity of weak
%solutions of degenerate elliptic equations}, Nonlinear Anal. \textbf{7}(1983), 827--850.

%\bibitem {du}Y. Du, \textit{Order structure and topological methods in
%nonlinear partial differential equations. Vol. 1. Maximum principles and
%applications}, World Scientific Publishing Co. Pte. Ltd., Hackensack, NJ, 2006.


%\bibitem {nodea}T. Godoy, U. Kaufmann, \textit{On strictly positive solutions
%for some semilinear elliptic problems}, NoDEA Nonlinear Differ. Equ. Appl.
%\textbf{20} (2013), 779-795.


%\bibitem {drabek}P. Dr\'{a}bek, J. Hern\'{a}ndez, \textit{Existence and
%uniqueness of positive solutions for some quasilinear elliptic problems},
%Nonlinear Anal. \textbf{44} (2001), 189--204.


%\bibitem {garcia}J. Garc\'{\i}a-Meli\'{a}n, J. Sabina de Lis, \textit{Maximum
%and comparison principles for operators involving the p-Laplacian}, J. Math
%Anal. Appl. \textbf{218} (1998), 49-65.

%\bibitem {GV}M. Guedda, L. Veron, \textit{Quasilinear elliptic equations
%involving critical Sobolev exponents}, Nonlinear Anal. \textbf{13 }(1989), 879--902.

%\bibitem {ans}T. Godoy, U. Kaufmann, \textit{Existence of strictly positive
%solutions for sublinear elliptic problems in bounded domains}, Adv. Nonlinear
%Stud. \textbf{14} (2014), 353--359.

%\bibitem {IO}T. Idogawa, M. Otani, \textit{The first eigenvalues of some
%abstract elliptic operators}, Funkcialaj Ekvacioj \textbf{38} (1995), 1--9.

%\bibitem {I}Y. Il'yasov, \textit{On positive solutions of indefinite elliptic
%equations, }Comptes Rendus de l'Acad\'{e}mie des Sciences-Series I-Mathematics
%\textbf{333} (2001), 533--538.

%\bibitem {bams}U. Kaufmann, I. Medri, \textit{Strictly positive solutions for
%one-dimensional nonlinear problems involving the p-Laplacian}, Bull. Aust.
%Math. Soc. \textbf{89} (2014), 243--251.


%\bibitem {jesusultimo}J. Hern\'{a}ndez, F. Mancebo, J. Vega, \textit{On the
%linearization of some singular, nonlinear elliptic problems and applications},
%Ann. Inst. H. Poincar\'{e} Anal. Non Lin\'{e}aire \textbf{19} (2002), 777--813.


%\bibitem {J}L. Jeanjean, \textit{Some continuation properties via minimax
%arguments}, Electron. J. Differential Equations \textbf{2011} (2011), Paper
%No. 48, 10 pp.


%\bibitem {K}R. Kajikiya, \textit{Positive solutions of semilinear elliptic
%equations with small perturbations}, Proc. Amer. Math. Soc. \textbf{141}
%(2013), 1335-1342.
\bibitem{KST} R. Kajikiya, I. Sim, S. Tanaka, {\it A complete classification of bifurcation diagrams for a class of (p,q)-Laplace equations.} J. Math. Anal. Appl. 462 (2018), no. 2, 1178--1194.

\bibitem {KRQU16}{U.\ Kaufmann, H.\ Ramos Quoirin, K.\ Umezu,
\textit{Positivity results for indefinite sublinear elliptic problems via a
continuity argument}, J.\ Differential Equations \textbf{263} (2017),
4481--4502. }

\bibitem {NoDEA}U.\ Kaufmann, H.\ Ramos Quoirin, K.\ Umezu{\small ,
}\textit{Positive solutions of an elliptic Neumann problem with a sublinear
indefinite nonlinearity}, NoDEA Nonlinear Differ. Equ. Appl. \textbf{25}
(2018), Art. 12, 34 pp.

\bibitem {KRQU3}U.\ Kaufmann, H.\ Ramos Quoirin, K.\ Umezu, \textit{A curve of
positive solutions for an indefinite sublinear Dirichlet problem}, Discrete
Contin. Dyn. Syst.\ \textbf{40} (2020), 617--645.

%\bibitem {KRQUR1}U.\ Kaufmann, H.\ Ramos Quoirin, K.\ Umezu,
%\textit{Nonnegative solutions of an indefinite sublinear Robin problem I:
%positivity, exact multiplicity, and existence of a subcontinuum}, Annali di
%Matematica (2020).\ https://doi.org/10.1007/s10231-020-00954-x

%\bibitem {rimut}U.\ Kaufmann, H.\ Ramos Quoirin, K.\ Umezu, \textit{Past and
%recent contributions to indefinite sublinear elliptic problems}, to appear in
%Rendiconti dell%
%TCIMACRO{\U{b4}}%
%BeginExpansion
%\'{}%
%EndExpansion
%Istituto di Matematica dell%
%TCIMACRO{\U{b4}}%
%BeginExpansion\'{}%
%EndExpansion Universit\`{a} di Trieste (2020).\ arXiv:2004.01284

%\bibitem {KRQUR2}U.\ Kaufmann, H.\ Ramos Quoirin, K.\ Umezu,
%\textit{Nonnegative solutions of an indefinite sublinear Robin problem II:
%local and global exactness results.} arXiv:2001.09315

\bibitem {KRQUpp}U.\ Kaufmann, H.\ Ramos Quoirin, K.\ Umezu,
\textit{Uniqueness and sign properties of minimizers in a quasilinear
indefinite problem}, arXiv:2001.11318.

\bibitem {KRQUpp2}U.\ Kaufmann, H.\ Ramos Quoirin, K.\ Umezu,
\textit{Uniqueness and positivity issues in a quasilinear indefinite problem},  	arXiv:2007.09498.

\bibitem {KLP}{B.\ Kawohl, M.\ Lucia, S.\ Prashanth, \textit{Simplicity of the
principal eigenvalue for indefinite quasilinear problems}, Adv. Differential
Equations, \textbf{12} (2007), 407--434. }

%\bibitem {KC}H. B. Keller, D. S. Cohen, \textit{Some positone problems in
%nonlinear heat conduction,} J. Math. Mech. \textbf{16} (1967), 1361--1376.

\bibitem {L}G. M. Lieberman, \textit{Boundary regularity for solutions of
degenerate elliptic equations}, Nonlinear Anal. \textbf{12} (1988),
1203--1219.

\bibitem{L2}  G. M. Lieberman, {\it The natural generalization of the natural conditions od Ladyzhen-skaya and Uraltseva for elliptic equations,} Comm. Partial Differential Equations 16 (1991), 311--361.
%\bibitem {LGMMT13}J. L\'{o}pez-G\'{o}mez, M. Molina-Meyer, A. Tellini,
%\textit{The uniqueness of the linearly stable positive solution for a class of
%superlinear indefinite problems with nonhomogeneous boundary conditions}, J.
%Differential Equations \textbf{255} (2013), 503--523.

\bibitem{MP} S. A. Marano, N. S. Papageorgiou  {\it Constant-sign and nodal solutions of coercive (p,q)- Laplacian problems}. Nonlinear Analysis: Theory, Methods and Applications, 77 (2013), 118--129.

\bibitem{MS} C. Morales-Rodrigo, A. Suárez, {\it Uniqueness of solution for elliptic problems with non-linear boundary conditions,} Comm. Appl. Nonlinear Anal. 13 (2006).

\bibitem{N} A.I. Nazarov,  {\it On the symmetry of extremals in the weight embedding theorem. Function theory and mathematical analysis.} J. Math. Sci. (New York) 107 (2001), no. 3, 3841--3859. 

%\bibitem {Ou}T. Ouyang, \textit{On the positive solutions of semilinear
%equations }$\Delta u+\lambda u+hu^{p}=0$\textit{ on compact manifolds II},
%Indiana Univ. Math. J. \textbf{40} (1991), 1083--1141.

%\bibitem {PT}M. A. Pozio and A. Tesei, \textit{Support properties of solution
%for a class of degenerate parabolic problems}, Comm. Partial Differential
%Equations \textbf{12} (1987), 47-75.

\bibitem{PS} P. Pucci and J. Serrin, {\it The Maximum Principle}, Birkhauser, Basel, 2007.
%\bibitem {RQU6}H. Ramos Quoirin, K. Umezu, \textit{An indefinite
%concave-convex equation under a Neumann boundary condition I},
%arXiv:1603.04940 (to appear in Israel J.\ Math.).

\bibitem{T} M. Tanaka, {\it Uniqueness of a positive solution and existence of a sign-changing solution for (p,q)-Laplace equation}. Journal of Nonlinear Functional Analysis, 2014 (2014), 1--15.

%\bibitem {tru}N. Trudinger, \textit{Linear elliptic operators with measurable
%coefficients}, Ann. Scuola Norm. Sup. Pisa \textbf{27} (1973), 265--308.


\bibitem {Va}J.L. V\'{a}zquez, {\it A strong maximum principle for some quasilinear
elliptic equations.} \textit{Appl. Math. Optim}\emph{.} \textbf{12} (1984), 191--202.
\end{thebibliography}
\end{document}